\documentclass[a4paper,10pt,leqno]{amsart}
\title{Localization, Whitehead groups, and the Atiyah Conjecture}

\author{Wolfgang L\"uck}
        \address{Mathematisches Institut der Universit\"at Bonn\\
                Endenicher Allee 60\\
                53115 Bonn, Germany}
         \email{wolfgang.lueck@him.uni-bonn.de}
          \urladdr{http://www.him.uni-bonn.de/lueck}

\author{Peter Linnell}
\address{Department of Mathematics\\
Virginia Tech\\
Blacksburg, VA 24061-0123,
\\USA}
\email{plinnell@math.vt.edu}
\urladdr{http://www.math.vt.edu/people/plinnell/}

         \date{February 2016}


\usepackage{hyperref}
\usepackage{color}
\usepackage{pdfsync}
\usepackage{calc}
\usepackage{enumerate,amssymb}
\usepackage[arrow,curve,matrix,tips,2cell]{xy}
  \SelectTips{eu}{10} \UseTips
  \UseAllTwocells

\DeclareMathAlphabet\EuR{U}{eur}{m}{n}
\SetMathAlphabet\EuR{bold}{U}{eur}{b}{n}

\makeindex             

\theoremstyle{plain}
\newtheorem{theorem}{Theorem}[section]
\newtheorem{lemma}[theorem]{Lemma}

\theoremstyle{definition}

\newtheorem{definition}[theorem]{Definition}

\newtheorem{remark}[theorem]{Remark}

{\catcode`@=11\global\let\c@equation=\c@theorem}

\newcommand{\comsquare}[8]                   
{\begin{CD}
#1 @>#2>> #3\\
@V{#4}VV @V{#5}VV\\
#6 @>#7>> #8
\end{CD}
}

\newcommand{\xycomsquare}[8]                   
{\xymatrix
{#1 \ar[r]^{#2} \ar[d]^{#4} &
#3 \ar[d]^{#5}  \\
#6\ar[r]^{#7} &
#8
}
}

\newcommand{\xycomsquareminus}[8]                      
{\xymatrix{#1 \ar[r]^-{#2} \ar[d]^-{#4} &
#3 \ar[d]^-{#5}  \\
#6\ar[r]^-{#7} &
#8
}
}

\newcommand{\cala}{\mathcal{A}}
\newcommand{\calb}{\mathcal{B}}
\newcommand{\calc}{\mathcal{C}}
\newcommand{\cald}{\mathcal{D}}

\newcommand{\caln}{{\mathcal N}}
\newcommand{\calr}{\mathcal{R}}

\newcommand{\calt}{\mathcal{T}}
\newcommand{\calu}{\mathcal{U}}
\newcommand{\calx}{\mathcal{X}}
\newcommand{\caly}{\mathcal{Y}}
\newcommand{\calz}{\mathcal{Z}}

\newcommand{\IC}{{\mathbb C}}

\newcommand{\IQ}{{\mathbb Q}}

\newcommand{\IZ}{{\mathbb Z}}

\newcommand{\coker}{\operatorname{coker}}

\newcommand{\id}{\operatorname{id}}

\newcommand{\Wh}{\operatorname{Wh}}

\DeclareMathOperator{\Loc}{L}

\newcommand{\higherlim}[3]{{\setbox1=\hbox{\rm lim}
        \setbox2=\hbox to \wd1{\leftarrowfill} \ht2=0pt \dp2=-1pt
        \mathop{\vtop{\baselineskip=5pt\box1\box2}}
        _{#1}}^{#2}#3}

\newcommand{\version}[1]                       
{\begin{center} last edited on #1\\
last compiled on \today\\
name of texfile: \jobname
\end{center}
}

\newcounter{commentcounter}

\begin{document}

\typeout{----------------------------  localization.tex  ----------------------------}


\typeout{------------------------------------ Abstract ----------------------------------------}

\begin{abstract}
  Let $K_1^w(\IZ G)$ be the $K_1$-group of square matrices over $\IZ G$ which are not
  necessarily invertible but induce weak isomorphisms after passing to Hilbert space
  completions.  Let $\cald(G;\IQ)$ be the division closure of $\IQ G$ in the algebra
  $\calu(G)$ of operators affiliated to the group von Neumann algebra. Let $\calc$ be the
  smallest class of groups which contains all free groups and is closed under directed
  unions and extensions with elementary amenable quotients.  Let $G$ be a torsionfree
  group which belongs to $\calc$. Then we prove that $K_1^w(\IZ (G)$ is isomorphic to
  $K_1(\cald(G;\IQ))$. Furthermore we show that $\cald(G;\IQ)$ is a skew field and hence
  $K_1(\cald(G;\IQ))$ is the abelianization of the multiplicative group of units in
  $\cald(G;\IQ)$.
\end{abstract}

\maketitle


 \typeout{-------------------------------   Section 0: Introduction --------------------------------}

\setcounter{section}{-1}
\section{Introduction}

In~\cite{Friedl-Lueck(2016l2_universal)} we have introduced the \emph{universal $L^2$-torsion}
$\rho^{(2)}_u(X;\caln(G))$ of an $L^2$-acyclic finite $G$-$CW$-complex $X$ and discussed its
applications.  It takes values in a certain abelian group $\Wh^w(G)$  which is the quotient of 
the $K_1$-group $K_1^w(\IZ G)$ by the subgroup
given by trivial units $\{\pm g \mid g \in G\}$. Elements $[A] \in K_1^w(\IZ G)$ are given 
by $(n,n)$-matrices $A$ over $\IZ G$ which are not necessarily
invertible but for which the operator $r_A^{(2)} \colon L^2(G)^n \to L^2(G)^n$ given by
right multiplication with $A$ is a weak isomorphism, i.e., it is injective and has dense
image.  We require for such  square matrices $A,B$ the  following relations in $K_1^w(\IZ G)$
\begin{eqnarray*}
[AB]  & = & [A] \cdot [B];
\\
\left[\begin{pmatrix}A & \ast \\ 0 & B \end{pmatrix}\right]
& = & [A] \cdot [B];
\end{eqnarray*}
More details about $\Wh^w(G)$ and $K_1^w(\IZ G)$ will be given in 
Section~\ref{sec:Proof_of_the_main_Theorem}.

Let $\cald(G;\IQ) \subseteq \calu(G)$  be the smallest subring of the
algebra $\calu(G)$ of operators $L^2(G) \to L^2(G)$ affiliated to the
group von Neumann algebra $\caln(G)$ which contains $\IQ G$ and is
division closed, i.e., any element in $\cald(G;\IQ)$ which is
invertible in $\calu(G)$ is already invertible in
$\cald(G;\IQ)$. (These notions will be explained in detail in
Subsection~\ref{subsec:Review_of_division_closure}.)
 
The main result of this paper is 

\begin{theorem}[$K_1^w(G)$ and units in $\cald(G;\IQ)$]
  \label{the:Whw(G)_and_units_in_cald(G;IQ)}
  Let $\calc$ be the smallest class of groups which contains all free groups and is closed
  under directed unions and extensions with elementary amenable quotients.  Let $G$ be a
  torsionfree group which belongs to $\calc$.

  Then $\cald(G;\IQ)$ is a skew field and there are isomorphisms
  \[
  K_1^w(\IZ G) \xrightarrow{\cong} K_1(\cald(G;\IQ)) \xrightarrow{\cong}
  \cald(G;\IQ)^{\times}/[\cald(G;\IQ)^{\times},\cald(G;\IQ)^{\times}].
  \]
\end{theorem}

In the special case that $G = \IZ$, the right side reduces to the
multiplicative abelian group of non-trivial elements in the field $\IQ(z,z^{-1})$ of
rational functions with rational coefficients in one variable.
This reflects the fact that in the case $G = \IZ$ the universal 
$L^2$-torsion is closely related to Alexander polynomials.


\subsection*{Acknowledgments.}

The first author was partially supported by a grant from the {NSA}.
The paper is financially supported by the Leibniz-Preis of
the second author granted by the {DFG} and the ERC Advanced Grant ``KL2MG-interactions''
(no.  662400) of the second author granted by the European Research Council.

\tableofcontents


 \typeout{----------------------------   Section 1: Universal localization  ---------------------------}

\section{Universal localization}
\label{sec:Universal_localization}


\subsection{Review of universal localization}
\label{subsec:Review_of_universal_localization}

Let $R$ be a (associative) ring (with unit) and let $\Sigma$ be a set of homomorphisms 
between finitely generated projective (left) $R$-modules.  A
ring homomorphism $f\colon R \to S$ is called \emph{$\Sigma$-inverting} if for every element
$\alpha\colon M \to N$ of $\Sigma$ the induced map $S \otimes_R \alpha\colon S\otimes_R
M\to S\otimes_R N$ is an isomorphism.  A $\Sigma$-inverting ring homomorphism $i\colon R
\to R_{\Sigma}$ is called \emph{universal $\Sigma$-inverting} if for any
$\Sigma$-inverting ring homomorphism $f\colon R \to S$ there is precisely one ring
homomorphism $f_{\Sigma}\colon R_{\Sigma} \to S$ satisfying $f_{\Sigma} \circ i = f$.  If
$f\colon R \to R_{\Sigma}$ and $f'\colon R \to R_{\Sigma}'$ are two universal
$\Sigma$-inverting homomorphisms, then by the universal property there is precisely one
isomorphism $g\colon R_{\Sigma} \to R_{\Sigma}'$ with $g \circ f = f'$.  This shows the
uniqueness of the universal $\Sigma$-inverting homomorphism. The universal
$\Sigma$-inverting ring homomorphism exists, see~\cite[Section 4]{Schofield(1985)}.  If $\Sigma$ is a
set of matrices, a model for $R_{\Sigma}$ is given by considering the free $R$-ring
generated by the set of symbols $\{\overline{a_{i,j}}\mid A = (a_{i,j}) \in \Sigma\}$ and
dividing out the relations given in matrix form by $\overline{A}A = A\overline{A} = 1$,
where $\overline{A}$ stands for $(\overline{a_{i,j}})$ for $A= (a_{i,j})$.  The map
$i\colon R \to R_{\Sigma}$ does not need to be injective and the functor 
$R_{\Sigma} \otimes_R -$ does not need to be exact in general.

A special case of a universal localization is the Ore localization $S^{-1}R$ of a ring $R$
for a multiplicative closed subset $S \subseteq R$ which satisfies the Ore condition, namely
take $\Sigma$ to be the set of $R$-homomorphisms $r_s \colon R \to R, \; r \mapsto rs$,
where $s$ runs through $S$.  For the Ore localization the functor $S^{-1}R \otimes_R -$
is exact and the kernel of the canonical map $R \to S^{-1}R$ is 
$\{r \in R \mid \exists s \in S \, \text{with}\;  rs = 0\}$.

Let $R$ be a ring and let $\Sigma$ be a set of homomorphisms between finitely generated projective  $R$-modules.  We
call $\Sigma$ \emph{saturated} if for any two elements $f_0 \colon P_0 \to Q_0$ and $f_1
\colon P_1 \to Q_1$ of $\Sigma$ and any $R$-homomorphism $g_0 \colon P_0 \to Q_1$ and $g_1
\colon P_1 \to Q_0$ the $R$-homomorphisms $\begin{pmatrix} f_0 & 0 \\g_0 &
  f_1 \end{pmatrix}\colon P_0 \oplus P_1 \to Q_0 \oplus Q_1$ and $\begin{pmatrix} f_0 &
  g_1 \\ 0 & f_1 \end{pmatrix}\colon P_0 \oplus P_1 \to Q_0 \oplus Q_1$ belong to $\Sigma$
and for every $R$-homomorphism $f_0 \colon P_0 \to Q_0$ which becomes invertible over
$R_{\Sigma}$, there is an element $f_1 \colon P_1 \to Q_1$ in $\Sigma$, finitely generated
projective $R$-modules $X$ and $Y$, and $R$-isomorphisms $u \colon P_0 \oplus X
\xrightarrow{\cong} P_1 \oplus Y$ and $v \colon Q_0 \oplus X \xrightarrow{\cong} Q_1 \oplus Y$ 
satisfying $(f_1 \oplus \id_Y) \circ u = v \circ (f_0 \oplus \id_X)$. We can
always find for $\Sigma$ another set $\Sigma'$ with $\Sigma \subseteq \Sigma'$ such that
$\Sigma'$ is saturated and the canonical map $R_{\Sigma} \to R_{\Sigma'}$ is bijective.
Moreover, in nearly all cases we will consider sets $\Sigma$ which
are already saturated.  Indeed if $\Sigma'$ denotes the set of all maps
between finitely generated projective (left) modules which become
invertible over $R_{\Sigma}$, then $\Sigma \subseteq \Sigma'$, $\Sigma'$ is saturated, and the canonical map
$R_{\Sigma} \to R_{\Sigma'}$ is an isomorphism
cf.~\cite[Exercise~7.2.8 on page~394]{Cohn(1985)}.
Therefore we can assume without harm in the sequel that $\Sigma$ is saturated.


\subsection{$K_1$ of universal localizations}
\label{subsec:K_1_of_universal_localization}

Let $R$ be a ring and let $\Sigma$ be a (saturated) set of homomorphisms between finitely
generated projective $R$-modules.
\begin{definition}[$K_1(R,\Sigma)$]
\label{def:K_1(R,Sigma)}
Let $K_1(R,\Sigma)$ be the abelian group defined in terms of generators and relations as
follows.  Generators $[f]$ are (conjugacy classes) of
$R$-endomorphisms $f \colon P \to P$
of finitely generated projective $R$-modules $P$ such that $\id_{R_{\Sigma}}\otimes_R f
\colon R_{\Sigma} \otimes_R P \to R_{\Sigma} \otimes_R P$ is an isomorphism.  If $f,g
\colon P \to P$ are $R$-endomorphisms of the same finitely generated projective $R$-module
$P$ such that $\id_{R_{\Sigma}}\otimes_R f$ and $\id_{R_{\Sigma}}\otimes_R g$ are
bijective, then we require the relation
\[
[g \circ f] = [g] + [f].
\]
If we have a commutative diagram of finitely generated projective $R$-modules
with exact rows
\[
\xymatrix{0 \ar[r] 
&
P_0 \ar[r]^{i} \ar[d]_{f_0}
& 
P_1 \ar[r]^{p} \ar[d]_{f_1}
& 
P_2 \ar[r] \ar[d]_{f_2}
&
0
\\
0 \ar[r] 
&
P_0 \ar[r]^{i} 
& 
P_1 \ar[r]^{p}
& 
P_2 \ar[r]
&
0
}
\]
such that $\id_{R_{\Sigma}}\otimes_R f_0$, $\id_{R_{\Sigma}}\otimes_R
f_2$ (and hence $\id_{R_{\Sigma}}\otimes_R f_1$)  
are bijective, then we require the relation
\[
[f_1] = [f_0] + [f_2].
\]
\end{definition}
If the set $\Sigma$ consists of all isomorphisms $R^n \xrightarrow{\cong} R^n$ for all $n
\ge 0$, then for an $R$-endomorphism $f \colon P \to P$ of a finitely generated projective
$R$-module $P$ the induced map $\id_{R_{\Sigma}} \otimes f$ is bijective if and only if
$f$ itself is already bijective and hence $K_1(R,\Sigma)$ is just the classical first
$K$-group $K_1(R)$.

The main result of this section is

\begin{theorem}[$K_1(R,\Sigma)$ and $K_1(R_{\Sigma})$]
\label{the:K_1(R,Sigma)_and_K_1(R_Sigma)}
Suppose that every element in $\Sigma$ is given by an endomorphism of a finitely generated
projective $R$-module and that the canonical map $i \colon R \to R_{\Sigma}$ is
injective. Then the homomorphism
\[
\alpha \colon K_1(R,\Sigma) \xrightarrow{\cong} K_1(R_{\Sigma}), \quad [f \colon P \to P] 
\mapsto [\id_{R_{\Sigma}} \otimes_R f \colon R_{\Sigma} \otimes_R P \to R_{\Sigma} \otimes_R P]
\]
is bijective. 
\end{theorem}
\begin{proof}
We construct an inverse
\begin{eqnarray}
& \beta \colon K_1(R_{\Sigma}) \to K_1(R,\Sigma) &
\label{beta_K_1(R_Sigma)_to_K_1(R,Sigma)}
\end{eqnarray}
as follows. Consider an element $x$ in $K_1(R_{\Sigma})$.  Then we can choose a finitely
generated projective $R$-module $Q$, (actually, we could choose it to be finitely
generated free), and an $R_{\Sigma}$-automorphism
\[
a \colon R_{\Sigma} \otimes_R Q \xrightarrow{\cong}  R_{\Sigma} \otimes_R Q
\]
such that $x = [a]$. Now the key ingredient is Cramer's rule, see~\cite[Theorem~4.3 on
page 53]{Schofield(1985)}.  It implies the existence of a finitely generated projective
$R$-module $P$, $R$-homomorphisms $b,b' \colon P \oplus Q \to P \oplus Q$ and a
$R_{\Sigma}$-homomorphism $a' \colon R_{\Sigma} \otimes_R Q \to R_{\Sigma} \otimes_R P$
such that $\id_{R_{\Sigma}} \otimes_R b$ is bijective, and for the
$R_{\Sigma}$-homomorphism $\begin{pmatrix} \id_{R_{\Sigma} \otimes_R P} & a' \\ 0 &
  a \end{pmatrix} \colon R_{\Sigma} \otimes_R P \oplus R_{\Sigma} \otimes_R Q \to
R_{\Sigma} \otimes_R P \oplus R_{\Sigma} \otimes_R Q$ the composite
\begin{multline*}
R_{\Sigma} \oplus (P \oplus Q) \xrightarrow{i} R_{\Sigma} \otimes_R  P \oplus R_{\Sigma} \otimes_R Q
\xrightarrow{\begin{pmatrix} \id_{R_{\Sigma} \otimes_R P}  & a' \\ 0 & a \end{pmatrix}} 
R_{\Sigma} \otimes_R  P \oplus R_{\Sigma} \otimes_R Q
\\
\xrightarrow{i^{-1}} R_{\Sigma} \oplus (P \oplus Q) 
\xrightarrow{\id_{R_{\Sigma}} \otimes_R b} R_{\Sigma} \oplus (P \oplus Q) 
\end{multline*}
agrees with $\id_{R_{\Sigma}}  \otimes_R b'$, where $i$ is the canonical $R_{\Sigma}$-isomorphism. 
Then also $\id_{R_{\Sigma}} \otimes_R b$ is bijective. We want to define
\begin{eqnarray}
\beta(x) & := & 
[b'] - [b].
\label{def_of_beta}
\end{eqnarray}
The main problem is to show that this is independent of the various choices. 
Given a  finitely generated projective $R$-module $P$ and an 
$R_{\Sigma}$-automorphism 
\[
a \colon R_{\Sigma} \otimes_R Q \xrightarrow{\cong}  R_{\Sigma} \otimes_R Q
\]
and two such choices $(P,b,b',a')$ and $(\overline{P},\overline{b},\overline{b}',\overline{a}')$, 
we show next
\begin{eqnarray}
[b] - [b] & := & 
[\overline{b}] - [\overline{b}'].
\label{def_of_beta_independence_of_(P;b,b',a')}
\end{eqnarray}
We can write
\begin{eqnarray*}
b 
& = & 
\begin{pmatrix} b_{P,P} 
&
b_{Q,P} \\ b_{P,Q} 
& 
b_{Q,Q} \end{pmatrix};
\\
b' 
& = & 
\begin{pmatrix} b'_{P,P} &b_{Q,P} \\ b'_{P,Q} & b'_{Q,Q} \end{pmatrix};
\\
\overline{b} 
& = & 
\begin{pmatrix} \overline{b}_{\overline{P},\overline{P}} 
&
\overline{b}_{Q,\overline{P}} \\ \overline{b}_{\overline{P},Q} 
& 
\overline{b}_{Q,Q} \end{pmatrix};
\\
\overline{b}' 
& = & 
\begin{pmatrix} \overline{b}'_{\overline{P},\overline{P}} 
&
\overline{b}_{Q,\overline{P}} \\ \overline{b}'_{\overline{P},Q} 
& 
\overline{b}'_{Q,Q} \end{pmatrix},
\end{eqnarray*}
for $R$-homomorphisms $b_{P,P} \colon P \to P$, $b_{P,Q} \colon P \to Q$, 
$b_{Q,P} \colon Q \to P$, $b_{Q,Q} \colon Q \to Q$, and analogously for $b'$, $\overline{b}$,
$\overline{b}'$. Then the relation between $b$ and $b'$ and $\overline{b}$ and
$\overline{b}'$ becomes
\[
\begin{pmatrix}
\id_{R_{\Sigma}} \otimes_R b_{P,P} 
& 
\id_{R_{\Sigma}} \otimes_R b_{Q,P}
\\
\id_{R_{\Sigma}} \otimes_R b_{P,Q}
&
\id_{R_{\Sigma}} \otimes_R b_{Q,Q}
\end{pmatrix}
\circ 
\begin{pmatrix}
\id_{R_{\Sigma} \otimes_R P}
& a'
\\
0
& 
a
\end{pmatrix}
 = 
\begin{pmatrix}
\id_{R_{\Sigma}} \otimes_R b'_{P,P} 
& 
\id_{R_{\Sigma}} \otimes_R b'_{Q,P}
\\
\id_{R_{\Sigma}} \otimes_R b'_{P,Q}
&
\id_{R_{\Sigma}} \otimes_R b'_{Q,Q}
\end{pmatrix}
\]
and analogously for $\overline{b}$ and $\overline{b}'$.  This implies 
$\id_{R_{\Sigma}}\otimes_R b_{P,P} = \id_{R_{\Sigma}} \otimes_R b'_{P,P}$ and hence 
$b_{P,P} = b'_{P,P}$ because of the injectivity of $i \colon R \to R_{\Sigma}$. Analogously 
we get  $b_{P,Q} = b'_{P,Q}$, $\overline{b}_{\overline{P},\overline{P}} =
\overline{b}'_{\overline{P},\overline{P}}$, and $\overline{b}_{\overline{P},Q} =
\overline{b}'_{\overline{P},Q}$.

The argument in~\cite[page~64-65]{Schofield(1985)} based on Macolmson's
criterion~\cite[Theorem~4.2 on page~53]{Schofield(1985)} implies that there exists
finitely generated projective $R$-modules $X_0$ and $X_1$, $R$-homomorphisms
\begin{eqnarray*}
d_1 \colon X_1 & \to & X_1,
\\
d_2 \colon  X_2 & \to & X_2,
\\
e_1 \colon X_1 & \to & Q,
\\
e_2 \colon X_2 & \to & P;
\\
\mu \colon P \oplus Q \oplus \overline{P} \oplus Q \oplus X_1 \oplus X_2 \oplus Q
& \to &
P \oplus Q \oplus \overline{P} \oplus Q \oplus  X_1 \oplus X_2 \oplus  Q;
\\
\nu \colon P \oplus Q \oplus \overline{P} \oplus Q \oplus X_1 \oplus X_2
&\to &
P \oplus Q \oplus \overline{P} \oplus Q \oplus  X_1 \oplus X_2;
\\
\tau \colon \colon P \oplus Q \oplus \overline{P} \oplus Q \oplus X_1 \oplus X_2 
& \to & 
Q,
\end{eqnarray*}
such that $\id_{R_{\Sigma}} \otimes_R d_1$, $\id_{R_{\Sigma}} \otimes_R d_2$, 
$\id_{R_{\Sigma}} \otimes_R \mu$ and $\id_{R_{\Sigma}} \otimes_R \nu$
are $R_{\Sigma}$-isomorphisms and for the four  $R$-homomorphisms 
\[P \oplus Q \oplus \overline{P} \oplus Q \oplus X_1 \oplus X_2 \oplus Q
\to
P \oplus Q \oplus \overline{P} \oplus Q \oplus  X_1 \oplus X_2 \oplus  Q
\]
given by 
\[
\alpha = \begin{pmatrix}
b_{P,P} & b_{Q,P} & 0 & 0 & 0 & 0 & 0
\\
b_{P,Q} & b_{Q,Q} & 0 & 0 & 0 & 0 & 0
\\
0 & \overline{b}_{Q,\overline{P}} & \overline{b}_{\overline{P},\overline{P}}'  & \overline{b}_{Q,\overline{P}}  & 0 & 0 & \overline{b}'_{Q,\overline{P}} 
\\
0 & \overline{b}_{Q,Q}  & \overline{b}_{\overline{P},Q}' & \overline{b}_{Q,Q}  & 0 & 0 &  \overline{b}'_{Q,Q}
\\
0 & 0 & 0 & 0 & d_1 & 0 & 0
\\
0 & 0 & 0 & 0 & 0 & d_2 & 0 
\\
0 & 0 & 0 & \id_{Q} & e_1 & 0 & 0
\end{pmatrix}
\]
\medskip
\[
\alpha' = \begin{pmatrix}
b_{P,P}' & b_{Q,P} & 0 & 0 & 0 & 0 & - b'_{Q,P}
\\
b_{P,Q}' & b_{Q,Q} & 0 & 0 & 0 & 0 & - b'_{Q,Q} 
\\
0 & \overline{b}_{Q,\overline{P}} & \overline{b}_{\overline{P},\overline{P}}  & \overline{b}_{Q,\overline{P}}  & 0 & 0 & 0
\\
0 & \overline{b}_{Q,Q}  & \overline{b}_{\overline{P},Q} & \overline{b}_{Q,Q}  & 0 & 0 &  0
\\
0 & 0 & 0 & 0 & d_1 & 0 & 0
\\
0 & 0 & 0 & 0 & 0 & d_2 & -e_2
\\
0 & 0 & 0 & \id_{Q} & e_1 & 0 & 0
\end{pmatrix}
\]
\medskip
\[
\gamma = \begin{pmatrix} \nu & 0 \\ 0 & \id_{Q} \end{pmatrix}
\]
and
\[
\gamma' = \begin{pmatrix} \nu & 0 \\ \tau
 & \id_{Q} \end{pmatrix}
\]
we get equations of maps of $R$-modules
\begin{eqnarray*}
\mu \circ \gamma & = & \alpha;
\\
\mu\circ \gamma' & = & \alpha'.
\end{eqnarray*}
Since $\id_{R_{\Sigma}} \otimes_R \mu$, $\id_{R_{\Sigma}} \otimes_R \gamma$ and
$\id_{R_{\Sigma}} \otimes_R\gamma'$ are isomorphisms, also $\id_{R_{\Sigma}} \otimes_R \alpha$ 
and $\id_{R_{\Sigma}} \otimes_R \alpha'$ are isomorphisms.  Hence we get
well-defined elements $[\mu]$, $[\nu]$, $[\nu']$, $[\alpha]$, and $[\alpha']$ in
$K_1(R,\Sigma)$ satisfying
\begin{eqnarray*}
{[\mu]} 
& = & 
{[\gamma]} + {[\alpha]};
\\
{[\mu]} 
& = &
{[\gamma']} + {[\alpha']};
\\
{[\gamma]}
& = & 
{[\gamma']}.
\end{eqnarray*}
This implies
\begin{eqnarray}
[\alpha] & = & [\alpha'].
\label{def_of_beta_independence_of_(P;b,b',a')_1}
\end{eqnarray}
If we interchange in the matrix defining $\alpha$ the fourth and the last column, we get a
matrix in a suitable block form which allows us to deduce
\begin{eqnarray}
[\alpha] 
& = & 
- \left[
\begin{pmatrix}
b_{P,P} & b_{Q,P} & 0 & 0 & 0 & 0 & 0
\\
b_{P,Q} & b_{Q,Q} & 0 & 0 & 0 & 0 & 0
\\
0 & \overline{b}_{Q,\overline{P}} & \overline{b}_{\overline{P},\overline{P}}'  & \overline{b}'_{Q,\overline{P}}  & 0 & 0 & \overline{b}_{Q,\overline{P}} 
\\
0 & \overline{b}_{Q,Q}  & \overline{b}_{\overline{P},Q}' & \overline{b}'_{Q,Q}  & 0 & 0 &  \overline{b}_{Q,Q}
\\
0 & 0 & 0 & 0 & d_1 & 0 & 0
\\
0 & 0 & 0 & 0 & 0 & d_2 & 0 
\\
0 & 0 & 0 & 0 & e_1 & 0 & \id_Q
\end{pmatrix}
\right]
\label{def_of_beta_independence_of_(P;b,b',a')_2}
\\
& = & 
- \left[\begin{pmatrix}
b_{P,P} & b_{Q,P} & 0 & 0 
\nonumber
\\
b_{P,Q} & b_{Q,Q} & 0 & 0 
\nonumber
\\
0 & \overline{b}_{Q,\overline{P}} & \overline{b}_{\overline{P},\overline{P}}'  & \overline{b}_{Q,\overline{P}}'  
\\
0 & \overline{b}_{Q,Q}  & \overline{b}_{\overline{P},Q}' &  \overline{b}'_{Q,Q}
\end{pmatrix}\right]
-
\left[\begin{pmatrix}d_1 & 0 & 0 \\ 0 & d_2 &  0 \\ e_1 & 0 & \id_Q
\end{pmatrix}\right]
\nonumber\\
& = & 
- \left[\begin{pmatrix} b_{P,P} & b_{Q,P} 
\\
b_{P,Q} & b_{Q,Q}
\end{pmatrix}
\right]
-
\left[\begin{pmatrix} \overline{b}'_{\overline{P},\overline{P}} & \overline{b}'_{Q,\overline{P}}
\nonumber
\\
\overline{b}'_{\overline{P},Q} & \overline{b}'_{Q,Q}
\end{pmatrix}\right]
- [d_1]- [d_2] - [\id_Q]
\nonumber
\\
& = & 
-[b] - [\overline{b}'] - [d_1] - [d_2].
\nonumber
\end{eqnarray}
Similarly we get  from the matrix describing $\alpha'$ after interchanging the
second and the last column, multiplying the second column with (-1),
interchanging the forth and the last column and finally subtracting appropriate multiples of the last
row from the third and row column to ensure that in the last column all entries except the
one in the right lower corner is trivial a matrix  in a suitable block form which allows us to deduce

\begin{eqnarray}
\quad \quad [\alpha' ]  & = & 
\left[\begin{pmatrix}
b_{P,P}' & b_{Q,P}' & 0 & 0 & 0 & 0 &  b_{Q,P}
\\
b_{P,Q}' & b_{Q,Q}' & 0 & 0 & 0 & 0 &  b_{Q,Q} 
\\
0 &  0& \overline{b}_{\overline{P},\overline{P}}  & \overline{b}_{Q,\overline{P}}  & 0 & 0 & \overline{b}_{Q,\overline{P}} 
\\
0 &  & \overline{b}_{\overline{P},Q} & \overline{b}_{Q,Q}  & 0 & 0 &  \overline{b}_{Q,Q} 
\\
0 & 0 & 0 & 0 & d_1 & 0 & 0
\\
0 & e_2 & 0 & 0 & 0 & d_2 & -e_2
\\
0 & 0 & 0 & \id_{Q} & e_1 & 0 & 0
\end{pmatrix}
\right]
\label{def_of_beta_independence_of_(P;b,b',a')_3}
\\
& = & 
\left[\begin{pmatrix}
b_{P,P}' & b_{Q,P}' & 0 & b_{Q,P} & 0 & 0 &  0
\\
b_{P,Q}' & b_{Q,Q}' & 0 & b_{Q,Q}  & 0 & 0 &  0
\\
0 &  0& \overline{b}_{\overline{P},\overline{P}}  & \overline{b}_{Q,\overline{P}}  & 0 & 0 & \overline{b}_{Q,\overline{P}} 
\\
0 &  & \overline{b}_{\overline{P},Q} & \overline{b}_{Q,Q}  & 0 & 0 &  \overline{b}_{Q,Q} 
\\
0 & 0 & 0 & 0 & d_1 & 0 & 0
\\
0 & e_2 & 0 & 0 & 0 & d_2 &0
\\
0 & 0 & 0 & 0 & e_1 & 0 & \id_{Q} 
\end{pmatrix}
\right]
\nonumber
\\
& = & 
- \left[\begin{pmatrix}
b_{P,P}' & b_{Q,P}' & 0 & b_{Q,P} & 0 & 0 &  0
\\
b_{P,Q}' & b_{Q,Q}' & 0 & b_{Q,Q}  & 0 & 0 &  0
\\
0 &  0& \overline{b}_{\overline{P},\overline{P}}  & \overline{b}_{Q,\overline{P}}  & - \overline{b}_{Q,\overline{P}}  \circ e_1 & 0 & 0
\\
0 &  & \overline{b}_{\overline{P},Q} & \overline{b}_{Q,Q}  &  - \overline{b}_{Q,Q} \circ e_1  & 0 &  0
\\
0 & 0 & 0 & 0 & d_1 & 0 & 0
\\
0 & e_2 & 0 & 0 & 0 & d_2 & 0
\\
0 & 0 & 0 & 0 & e_1 & 0 & \id_{Q} 
\end{pmatrix}
\right]
\nonumber
\\
& = & 
- \left[\begin{pmatrix}
b_{P,P}' & b_{Q,P}' & 0 & b_{Q,P} & 0 
\\
b_{P,Q}' & b_{Q,Q}' & 0 & b_{Q,Q}  & 0 
\\
0 &  0 & \overline{b}_{\overline{P},\overline{P}}  & \overline{b}_{Q,\overline{P}}  
& - \overline{b}_{Q,\overline{P}}  \circ e_1 
\\
0 &  & \overline{b}_{\overline{P},Q} & \overline{b}_{Q,Q}  &  - \overline{b}_{Q,Q} \circ e_1  
\\
0 & 0 & 0 & 0 & d_1 
\end{pmatrix}
\right]
- \left[\begin{pmatrix}
d_2 & - e_2 \\ 0 & \id_Q
\end{pmatrix}
\right]
\nonumber
\\
& = & 
- \left[\begin{pmatrix}
b_{P,P}' & b_{Q,P}' & 0 & b_{Q,P} 
\\
b_{P,Q}' & b_{Q,Q}' & 0 & b_{Q,Q}  
\\
0 &  0& \overline{b}_{\overline{P},\overline{P}}  & \overline{b}_{Q,\overline{P}}  
\\
0 &  & \overline{b}_{\overline{P},Q} & \overline{b}_{Q,Q}  
\end{pmatrix}
\right]
- [d_1] 
- [d_2] 
-[\id_Q]
\nonumber
\\
& = & 
- \left[\begin{pmatrix}
b_{P,P}' & b_{Q,P}' 
\\
b_{P,Q}' & b_{Q,Q}' 
\end{pmatrix}
\right]
-
 \left[\begin{pmatrix}
\overline{b}_{\overline{P},\overline{P}}  & \overline{b}_{Q,\overline{P}} 
\\
\overline{b}_{\overline{P},Q} & \overline{b}_{Q,Q} 
\end{pmatrix}
\right]
- [d_1] 
- [d_2] 
\nonumber
\\
& = & 
- [b'] -  [\overline{b}] - [d_1] - [d_2].
\nonumber
\end{eqnarray}
Now~\eqref{def_of_beta_independence_of_(P;b,b',a')} follows from 
equations~\eqref{def_of_beta_independence_of_(P;b,b',a')_1},~\eqref{def_of_beta_independence_of_(P;b,b',a')_2},
and~\eqref{def_of_beta_independence_of_(P;b,b',a')_3}.

We conclude from~\eqref{def_of_beta_independence_of_(P;b,b',a')_3} that we 
can assign to a finitely generated projective $R$-module $P$ and  an $R_{\Sigma}$-automorphism 
$a \colon R_{\Sigma} \otimes_R Q \xrightarrow{\cong}  R_{\Sigma} \otimes_R Q$
a well-defined element
\begin{eqnarray}
[a] \in K_1(R,\Sigma).
\label{[a]_in_K_1(R,sigma)}
\end{eqnarray}
If we have an isomorphism $u \colon Q \xrightarrow{\cong}  Q'$ of finitely generated 
projective $R$-modules, then one easily checks
\begin{eqnarray}
[(\id_{R_{\Sigma}} \otimes_R u) \circ a \circ (\id_{R_{\Sigma}} \otimes_R u)^{-1}] = [a].
\label{conjugation_invariance_of_[a|}
\end{eqnarray}
Given two finitely generated projective $R$-modules $Q$ and $\overline{Q}$ and
$R_{\Sigma}$-automorphisms 
$a \colon R_{\Sigma} \otimes_R Q \xrightarrow{\cong}  R_{\Sigma} \otimes_R Q$ and
$\overline{a}\colon R_{\Sigma} \otimes_R \overline{Q} 
\xrightarrow{\cong}  R_{\Sigma} \otimes_R \overline{Q}$,
one easily checks
\begin{eqnarray}
[a \oplus \overline{a}] = [a] + [\overline{a}].
\label{direct_sum_and_[a|}
\end{eqnarray}
Obviously we get for any finitely generated projective $R$-module $Q$
\begin{eqnarray}
[(\id_{R_{\Sigma}} \otimes_R \id_Q)] & = 0.
\label{[id]_gives_zero}
\end{eqnarray}
Consider a finitely generated projective $R$-module $Q$ and two $R_{\Sigma}$-isomorphisms
$a, \overline{a} \colon R_{\Sigma} \otimes_R Q \xrightarrow{\cong} R_{\Sigma} \otimes_R
Q$.  Next we want to show
\begin{eqnarray}
[\overline{a} \circ a] & = & [\overline{a} ] + [a].
\label{[a'a]_is_[a']_plus_[a]}
\end{eqnarray}
Make the choices $(P,b,b',a')$ and
$(\overline{P},\overline{b},\overline{b}',\overline{a}')$ for $a$ and $\overline{a}$ as we
did above in the definition of $[a]$ and $[\overline{a}]$.  Consider the
$R_{\Sigma}$-automorphism
\[
A =
\begin{pmatrix} \id_{R_{\sigma} \otimes_R P} & 0 & 0 & a'
\\
0 & \id_{R_{\sigma} \otimes_R Q}  & 0 & a
\\
0 & 0 & \id_{R_{\sigma} \otimes_R \overline{P}} & \overline{a}'a
\\
0 & 0 & 0 & \overline{a}a
\end{pmatrix}
\]
of $(R_{\Sigma} \otimes_R P) \oplus (R_{\Sigma} \otimes_R Q)  \oplus (R_{\Sigma} \otimes_R \overline{P})  
\oplus (R_{\Sigma} \otimes_R Q)$,
and the $R$-endomorphisms of $P \oplus Q \oplus \overline{P} \oplus Q$
\[
B = 
\begin{pmatrix}
b_{P,P} & b_{Q,P} & 0 & 0
\\
b_{P,Q} & b_{Q,Q} & 0 & 0
\\
0 &  - \overline{b}'_{Q,\overline{P}} & \overline{b}_{\overline{P},\overline{P}} & \overline{b}_{Q,\overline{P}}
\\
0 &  - \overline{b}'_{Q,Q} & \overline{b}_{\overline{P},Q} & \overline{b}_{Q,Q}
\end{pmatrix}
\]
and 
\[
B' = 
\begin{pmatrix}
b_{P,P}' & b_{Q,P} & 0 & b_{Q,P}' 
\\
b_{P,Q}' & b_{Q,Q} & 0 & b_{Q,Q}'
\\
0 &  - \overline{b}'_{Q,P} & \overline{b}_{P,P} & 0
\\
0 &  - \overline{b}'_{Q,Q} & \overline{b}_{P,Q} & 0
\end{pmatrix}
\]
From the block structure of $B$ one concludes that $(\id_{R_{\Sigma}} \otimes B)$ is 
an isomorphism and we get in $K_1(R,\Sigma)$
\begin{eqnarray}
[B]
& = & 
\left[\begin{pmatrix}
b_{P,P} & b_{Q,P} 
\\
b_{P,Q} & b_{Q,Q} 
\end{pmatrix}\right]
+
\left[\begin{pmatrix}
\overline{b}_{P,P} & \overline{b}_{Q,P}
\\
\overline{b}_{P,Q} & \overline{b}_{Q,Q}
\end{pmatrix}\right]
\label{[a'a]_is_[a']_plus_[a]_1}
\\
& = & 
[b] +  [\overline{b}].
\nonumber
\end{eqnarray}
If interchange in $B''$ the second and last column and multiply the last column with $-1$,
we conclude from the block structure of the resulting matrix that $(\id_{R_{\Sigma}}
\otimes B')$ is an isomorphism and we get in $K_1(R,\Sigma)$
\begin{eqnarray}
[B']
& = &
\left[
\begin{pmatrix}
b_{P,P}' & b_{Q,P}' & 0 & b_{Q,P} 
\\
b_{P,Q}' & b_{Q,Q}' & 0 & b_{Q,Q}
\\
0 &   0& \overline{b}_{\overline{P},\overline{P}} &  \overline{b}'_{Q,\overline{P}} 
\\
0 &  0& \overline{b}_{\overline{P},Q} &   \overline{b}'_{Q,Q} 
\end{pmatrix}
\right]
\label{[a'a]_is_[a']_plus_[a]_2}
\\
& = & 
\left[
\begin{pmatrix}
b_{P,P}' & b_{Q,P}' 
\\
b_{P,Q}' & b_{Q,Q}' 
\end{pmatrix}
\right]
+
\left[
\begin{pmatrix}
\overline{b}_{P,P} &  \overline{b}'_{Q,P} 
\nonumber
\\
\overline{b}_{P,Q} &   \overline{b}'_{Q,Q} 
\end{pmatrix}
\right]
\\
& = &
[b'] + [\overline{b'}].
\nonumber 
\end{eqnarray}
Since $(\id_{R_{\Sigma}} \otimes B)$ and $(\id_{R_{\Sigma}} \otimes B')$ are isomorphism
and we have $(\id_{R_{\Sigma}} \otimes B) \circ A = (\id_{R_{\Sigma}} \otimes B')$, we get
directly from the definitions
\begin{eqnarray}
[\overline{a}a] 
& = &
[B'] - [B].
\label{[a'a]_is_[a']_plus_[a]_3}
\end{eqnarray}
Now equation~\eqref{[a'a]_is_[a']_plus_[a]} follows from
equations~\eqref{[a'a]_is_[a']_plus_[a]_1},~\eqref{[a'a]_is_[a']_plus_[a]_2}, 
and~\eqref{[a'a]_is_[a']_plus_[a]_3}.
Now one easily checks that equations~\eqref{conjugation_invariance_of_[a|},%
~\eqref{direct_sum_and_[a|},~\eqref{[id]_gives_zero} and 
\eqref{[a'a]_is_[a']_plus_[a]} imply that the homomorphism $\beta$ 
announced in~\eqref{beta_K_1(R_Sigma)_to_K_1(R,Sigma)} is well-defined.
One easily checks that $\beta$ is an inverse to the homomorphism
$\alpha$ appearing in the statement of Theorem~\ref{the:K_1(R,Sigma)_and_K_1(R_Sigma)}.
This finishes the proof of Theorem~\ref{the:K_1(R,Sigma)_and_K_1(R_Sigma)}.
\end{proof}


\subsection{Schofield's localization sequence}
\label{subsec:Schofields_localization_sequence}

The proofs of this paper are motivated by Schofield's construction of a localization
sequence
\[
K_1(R) \to K_1(R_{\Sigma}) \to K_1(\calt) \to K_0(R) \to K_0(R_{\Sigma})
\]
where $\calt$ is the full subcategory of the category of the finitely presented $R$-modules
whose objects are cokernels of elements in $\Sigma$, 
see~\cite[Theorem~5.12 on page~60]{Schofield(1985)}.  
Under certain conditions this sequence has been extended to the
left in~\cite{Neeman(2007),Neeman-Ranicki(2004)}.  Notice that in connection with
potential proofs of the Atiyah Conjecture it is important to figure out under which
condition $K_0(FG) \to K_0(\cald(G;F))$ is surjective for a torsionfree group $G$ and a
subfield $F \subseteq \IC$, see~\cite[Theorem~10.38 on page~387]{Lueck(2002)}.  In this
connection the question becomes interesting whether $G$ has property (UL), see
Subsection~\ref{subsec:The_property_(UL)}, and how to continue the sequence above to the right.


\typeout{--------------------------   Section 2: Groups with property ULA ---------------------------}

\section{Groups with property (ULA)}
\label{sec:Groups_with_property_(ULA)}

Throughout this section let $F$ be a field with $\IQ \subseteq F \subseteq \IC$.


\subsection{Review of division and rational closure}
\label{subsec:Review_of_division_closure}

Let $R$ be a subring of the ring $S$. The \emph{division closure} $\cald(R \subseteq S) \subseteq S$ 
is the smallest subring of $S$ which contains $R$ and is division closed,
i.e., any element $x \in \cald(R \subset S)$ which is invertible in $S$ is already invertible in
$\cald(R \subseteq S)$. The \emph{rational closure} $\calr(R \subseteq S) \subseteq S$ is
the smallest subring of $S$ which contains $R$ and is rationally
closed, i.e., for every
natural number $n$ and matrix $A \in M_{n,n}(\cald(R \subseteq S))$ which is invertible in
$S$, the matrix $A$  is already invertible over $\calr(R \subseteq S)$. The division closure and the
rational closure always exist. Obviously 
$R \subseteq \cald(R \subseteq S) \subseteq \calr(R \subseteq S) \subseteq S$.

Consider an inclusion of rings $R \subseteq S$. Let $\Sigma(R \subseteq S)$ the set of all
square matrices over $R$ which become invertible over $S$.  Then there is a canonical
epimorphism of rings from the universal localization of $R$ with respect to 
$\Sigma(R \subseteq S)$ to the rational closure of $R$ in $S$,
see~\cite[Proposition~4.10 (iii)]{Reich(2006)}
\begin{eqnarray}
\lambda \colon R_{\Sigma(R \subseteq S)} \to \calr(R \subseteq S).
\label{lambda_colon_R_Sigma_to_calr}
\end{eqnarray}
Recall that we have inclusions $R \subseteq \cald(R \subseteq S) \to \calr(R \subseteq S)\subseteq S$.

Consider a group $G$. Let $\caln(G)$ be the group von Neumann algebra which can be
identified with the algebra $\calb(L^2(G),L^2(G))^G$ of bounded $G$-equivariant operators
$L^2(G) \to L^2(G)$.  Denote by $\calu(G)$ the algebra of operators which are affiliated
to the group von Neumann algebra.  This is the same as the Ore localization of $\caln(G)$
with respect to the multiplicatively closed subset of non-zero divisors in $\caln(G)$,
see~\cite[Chapter~8]{Lueck(2002)}. By the right regular representation we can embed 
$\IC G$ and hence also $FG$ as a subring in $\caln(G)$.  We will denote by $\calr(G;F)$ and
$\cald(G;F)$ the division and the rational closure of $FG$ in $\calu(G)$.  So we get a
commutative diagram of inclusions of rings
\[
\xymatrix@!C= 8em{
FG \ar[r] \ar[d]
&
\caln(G) \ar[dd]
\\
\cald(G;F) \ar[d]
&
\\
\calr(G;F) \ar[r]
& \calu(G)}
\]


\subsection{Review of the Atiyah Conjecture for torsionfree groups}
\label{subsec:Review_of_the_Atiyah_Conjecture_for_torsionfree_groups}

Recall that there is a dimension function $\dim_{\caln(G)}$ defined for all (algebraic) $\caln(G)$-modules,
see~\cite[Section~6.1]{Lueck(2002)}.

\begin{definition}[Atiyah Conjecture with coefficients in $F$]
  \label{def:Atiyah_Conjecture_with_coefficients_in_F}
  We say that a torsionfree group $G$ satisfies the \emph{Atiyah Conjecture with
    coefficients in $F$} if for any matrix $A \in M_{m,n}(FG)$ the von Neumann dimension
  $\dim_{\caln(G)}(\ker(r_A))$ of the kernel of the $\caln(G)$-homomorphism $r_A \colon
  \caln(G)^m \to \caln(G)^n$ given by right multiplication with $A$ is an integer.
\end{definition}

\begin{theorem}[Status of the Atiyah Conjecture]
\label{the:Status_of_the_Atiyah_Conjecture}\
\
\begin{enumerate}

\item \label{the:Status_of_the_Atiyah_Conjecture:subgroups}
If the torsionfree group $G$ satisfies the Atiyah Conjecture with coefficients in $F$, 
then also each of its subgroups satisfy the Atiyah Conjecture with coefficients in $F$;

\item \label{the:Status_of_the_Atiyah_Conjecture:subfields}
If the torsionfree group $G$ satisfies the Atiyah Conjecture with coefficients in $\IC$, 
then $G$ satisfies the Atiyah Conjecture with coefficients in $F$;

\item \label{the:Status_of_the_Atiyah_Conjecture:skew_field}
The torsionfree group $G$ satisfies the Atiyah Conjecture with coefficients in $F$ if and only if
$\cald(G;F)$ is a skew field;

\label{the:Status_of_the_Atiyah_Conjecture:cald_is_calr}
If the torsionfree group $G$ satisfies the Atiyah Conjecture with coefficients in $F$, then
the rational closure $\calr(G;F)$ agrees with the division closure $\cald(G;F)$;

\item \label{the:Status_of_the_Atiyah_Conjecture:Linnell} Let $\calc$ be the smallest
  class of groups which contains all free groups and is closed under directed unions and
  extensions with elementary amenable quotients.  Suppose that $G$ is a torsionfree group
  which belongs to $\calc$.

Then $G$ satisfies the Atiyah Conjecture with coefficients in $\IC$;

\item \label{the:Status_of_the_Atiyah_Conjecture:3-manifold_not_graph} Let $G$ be an infinite group which is the
  fundamental group of a compact connected  orientable irreducible $3$-manifold $M$ with empty or toroidal
  boundary. Suppose that one of the following conditions is satisfied:

 \begin{itemize}
  \item $M$ is not a closed graph manifold;
  \item $M$ is a closed graph manifold which admits a Riemannian metric of  non-positive sectional curvature.
  \end{itemize}
  
 Then $G$ is torsionfree and belongs to $\calc$. In particular $G$ satisfies the Atiyah Conjecture with coefficients in $\IC$;

\item \label{the:Status_of_the_Atiyah_Conjecture:approx}
Let $\cald$ be the smallest  class of groups such that
\begin{itemize}

\item The trivial group belongs to $\cald$;

\item If $p\colon G \to A$ is an epimorphism of a torsionfree group $G$ onto an
elementary amenable group $A$ and if $p^{-1}(B) \in \cald$ for every finite group
$B \subset A$, then $G \in \cald$;

\item $\cald$ is closed under taking subgroups;

\item $\cald$ is closed under colimits and inverse limits over directed systems.

\end{itemize}

If the group $G$ belongs to $\cald$,
then $G$ is torsionfree and the Atiyah Conjecture with coefficients in $\overline{\IQ}$  holds for $G$.

The class $\cald$ is closed under direct sums, direct products and free
products. Every residually torsionfree elementary amenable group
belongs to $\cald$;

\end{enumerate}

\end{theorem}
\begin{proof}~\eqref{the:Status_of_the_Atiyah_Conjecture:subgroups} 
  This follows from~\cite[Theorem~6.29~(2) on page~253]{Lueck(2002)}.
  \\[1mm]~\eqref{the:Status_of_the_Atiyah_Conjecture:subfields} This is obvious.
  \\[1mm]\eqref{the:Status_of_the_Atiyah_Conjecture:skew_field} This is proved in the
  case $F = \IC$ in~\cite[Lemma~10.39 on page~388]{Lueck(2002)}. The proof goes through
  for an arbitrary field $F$ with $\IQ \subseteq F \subseteq \IC$ without   modifications. 
  \\[1mm]~\eqref{the:Status_of_the_Atiyah_Conjecture:Linnell} This is due to Linnell,
  see for instance~\cite{Linnell(1993)} or~\cite[Theorem~10.19 on page~378]{Lueck(2002)}.
  \\[1mm]~\eqref{the:Status_of_the_Atiyah_Conjecture:3-manifold_not_graph} It suffices to
  show that $G = \pi_1(M)$ belongs to the class $\calc$ appearing in
  assertion~\eqref{the:Status_of_the_Atiyah_Conjecture:Linnell}.  As explained
  in~\cite[Section~10]{Dubois-Friedl-Lueck(2014Alexander)}, we conclude from
  combining papers by Agol, Liu, Przytycki-Wise, and 
  Wise~\cite{Agol(2008),Agol(2013),Liu(2013),Przytycki-Wise(2012), Przytycki-Wise(2014),Wise(2012raggs),Wise(2012hierachy)}
  that there exists a finite normal covering $p \colon \overline{M} \to M$ and a fiber
  bundle $S \to \overline{M} \to S^1$ for some compact connected orientable surface $S$. Hence it
  suffices to show that $\pi_1(S)$ belongs to $\calc$.  If $S$ has non-empty boundary,
  this follows from the fact that $\pi_1(S)$ is free. If $S$ is closed, the commutator
  subgroup of $\pi_1(S)$ is free and hence $\pi_1(S)$ belongs to $\calc$. Now
  assertion~\eqref{the:Status_of_the_Atiyah_Conjecture:3-manifold_not_graph} follows from
  assertion~\eqref{the:Status_of_the_Atiyah_Conjecture:Linnell}.
  \\[1mm]~\eqref{the:Status_of_the_Atiyah_Conjecture:approx} This result is due to Schick for $\IQ$
  see for instance~\cite{Schick(2001b)} or~\cite[Theorem~10.22 on page~379]{Lueck(2002)}
and for $\overline{\IQ}$ due to 
Dodziuk-Linnell-Mathai-Schick-Yates~\cite[Theorem~1.4]{Dodziuk-Linnell-Mathai-Schick_Yates(2003)}
  \end{proof}

For more information and further explanations about the Atiyah Conjecture
we refer for instance to~\cite[Chapter~10]{Lueck(2002)}.


\subsection{The property (UL)}
\label{subsec:The_property_(UL)}

\begin{definition}[Property (UL)]\label{def:property:(UL)}
We say that a group $G$ has the property (UL) with respect to $F$, if the canonical epimorphism
\[
\lambda \colon FG_{\Sigma(FG \subseteq \calu(G,F))} \to \calr(G;F)
\]
defined in~\eqref{lambda_colon_R_Sigma_to_calr} is bijective.
\end{definition}

Next we investigate which groups $G$ are known to have property (UL).

Let $\cala$ denote the class of groups consisting of the
finitely generated free groups and the amenable groups.
If $\caly$ and $\calz$ are classes of groups,
define $\Loc(\caly)
= \{G \mid \text{ every finite subset of }G \text{ is
contained in a } \caly \text{-group}\}$, and
$\caly\calz = \{G \mid \text{ there exists } H\lhd G
\text{ such that } H \in \caly \text{ and } G/H \in
\calz\}$.  Now define
$\calx$ to be the smallest class of groups which contains
$\cala$ and is closed under directed unions and group
extension.  Next for each ordinal $a$, define a class of groups
$\calx_a$ as follows:
\begin{itemize}
\item
$\calx_0 = \{1\}$.

\item
$\calx_a = \Loc(\calx_{a - 1}\cala)$ if $a$ is
a successor ordinal.

\item
$\calx_a = \bigcup_{b < a} \calx_b$ if $a$ is a
limit ordinal.
\end{itemize}

\begin{lemma} \label{lem:Ldescription}\
\begin{enumerate}
\item \label{lem:Ldescription:a}
Each $\calx_a$ is subgroup closed.

\item \label{lem:Ldescription:b}
$\calx = \bigcup_{a \ge 0} \calx_a$.

\item \label{lem:Ldescription:c}
$\calx$ is subgroup closed.
\end{enumerate}

\end{lemma}
\begin{proof}~\eqref{lem:Ldescription:a}
This is easily proved by induction on $a$.
\\[1mm]~\eqref{lem:Ldescription:b}
Set $\caly = \bigcup_{a \ge 0} \calx_a$.
Obviously $\calx \supseteq \caly$.
We prove the reverse inclusion by showing that $\caly$ is
closed under directed unions and group extension.  The former is
obvious, 
because if the group $G$ is the directed union of subgroups $G_i$ and
$a_i$ is the least ordinal such that $G_i \in
\calx_{a_i}$, we set $a = \sup_i a_i$ and
then $G \in \calx_{a + 1}$.
For the latter, we show that $\calx_a
\calx_b \subseteq \calx_{a+b}$ by induction on $b$, the
case $b=0$ being obvious.  If $b$ is a successor ordinal, write $b =
c+1$.  Then
\begin{align*}
\calx_a \calx_b &= \calx_a
(\Loc({X}_c) \cala) \subseteq \Loc (\calx_a\calx_c)
\cala\\
&\subseteq \Loc(\calx_{a+c}) \cala \quad \text{by
induction}\\
&\subseteq \calx_{a+c+1} = \calx_{a+b}.
\end{align*}
On the other hand, if $b$ is a limit ordinal, then
\begin{align*}
\calx_a \calx_b &= \calx_a \left( \bigcup_{c<b}
\calx_c \right) = \bigcup_{c<b} \calx_a \calx_c\\
&\subseteq \bigcup_{c<b} \calx_{a+c} \quad \text{by
induction}\\
&\subseteq \calx_{a+b}
\end{align*}
as required.
\\[1mm]~\eqref{lem:Ldescription:c}
This follows from assertions~\eqref{lem:Ldescription:a} and~\eqref{lem:Ldescription:b}.
\end{proof}

\begin{lemma} \label{lem:Llocally}
Let $G = \bigcup_{i \in I} G_i$ be groups such that given
$i,j \in I$, there exists $l \in I$ such that
$G_i,G_j \subseteq G_l$.  Write $\Sigma = \Sigma(FG \subseteq \calu(G))$ and
$\Sigma_i = \Sigma(FG_i \subseteq \calu(G_i))$ for $i \in I$.  Suppose the
identity map on $FG_i$ extends to an isomorphism $\lambda_i \colon
(FG_i)_{\Sigma_i}
\to \calr(G_i;F)$ for all $i\in I$.  

Then the identity map on $FG$ extends to an isomorphism 
$\lambda \colon FG_{\Sigma} \to \calr(G;F)$.
\end{lemma}
\begin{proof}
  By definition, the identity map on $FG$ extends to an epimorphism 
  $\lambda \colon  FG_{\Sigma} \to \calr(G;F)$.  We need to show that $\lambda$ is injective, 
  and here we   follow the proof of~\cite[Lemma~13.5]{Linnell(1998)}.  Clearly 
  $\Sigma_i \subseteq   \Sigma$ for all $i\in I$ and thus the inclusion map 
  $FG_i \hookrightarrow  FG$ extends to a map $\mu_i \colon (FG_i)_{\Sigma_i} \to FG_{\Sigma}$ for all 
  $i \in  I$.  Since $\lambda_i$ is an isomorphism, we may define
  $\nu_i =\mu_i \circ \lambda_i^{-1}
\colon \calr(G_i;F) \to FG_{\Sigma}$ for all $i\in I$.
If $G_i \subseteq G_j$, then $\calr(G_i;F)\subseteq  \calr(G_j;F)$ 
and we let $\psi_{ij} \colon \calr(G_i;F)\to \calr(G_j;F)$ denote the natural
  inclusion.  Observe that
$\mu_i(x) = \mu_j \lambda_j^{-1} \psi_{ij} \lambda_i(x)$ for all $x$
in the image of $FG_i$ in $(FG_i)_{\Sigma_i}$ and therefore by the
universal property, $\mu_i = \mu_j \lambda_j^{-1} \psi_{ij} \lambda_i$ and
hence $\mu_i \lambda_i^{-1} = \mu_j \lambda_j^{-1} \psi_{ij}$.  Thus
$\nu_i = \nu_j \psi_{ij}$ and the $\nu_i$ fit together to give a map
$\nu \colon \bigcup_{i \in I} \calr(G_i;F)\to FG_{\Sigma}$.  It is
easily checked that
$\nu \circ \lambda \colon FG_{\Sigma} \to FG_{\Sigma}$ is a map which
is the identity on the image of $FG$ in $FG_{\Sigma}$
and hence by the universal property of localization, $\nu \circ \lambda$ is the
  identity.  This proves that $\lambda$ is
injective, as required.
\end{proof}

If $G$ is a group and $\alpha$ is an automorphism of $G$, then
$\alpha$ extends to an automorphism of $\calu(G)$, which we
shall also denote by $\alpha$.  This is not only an algebraic
automorphism, but is also a homeomorphism with respect to the various
topologies on $\calu(G)$.

\begin{lemma} \label{lem:Lauto}
If  $\alpha$ is an automorphism of $G$, then $\alpha(\cald(G;F)) = \cald(G;F)$.
\end{lemma}
\begin{proof}
This is clear, because $\alpha(FG) = FG$.
\end{proof}

\begin{lemma} \label{lem:Lcrossedproduct} Let $H\lhd G$ be groups and let $\cald(H;F)G$
  denote the subring of $\cald(G;F)$ generated by $\cald(H;F)$ and $G$.  

  Then for a   suitable crossed product, $\cald(H;F) G \cong \cald(H;F)*G/H$ by a map 
  which extends the identity on $\cald(H;F)$ and for $g \in G$ sends
$\cald(H;F)\cdot g$ to $\cald(H;F) * Hg$.
\end{lemma}
\begin{proof}
  Let $T$ be a transversal for $H$ in $G$.  Since $h \mapsto tht^{-1}$ is an
  automorphism of $H$, we see that $t\cdot \cald(H;F)\cdot t^{-1} = \cald(H;F)$ for all $t \in T$ by
  Lemma~\ref{lem:Lauto} and so $\cald(H;F) G = \sum_{t\in T}\cald(H;F) G \cdot t$.  This sum is
  direct because the sum $\sum_{t\in T} \calu(H) \cdot t$ is direct and the result is
  established.
\end{proof}

In the sequel recall that $\calr(G;F) = \cald(G;F)$ holds if $\cald(G;F)$ is a skew field.

\begin{lemma} \label{lem:Lfiniteindex} Let $H \lhd G$ be groups such that $G/H$ is finite
  and $H$ is torsion free.  Assume that  $\cald(H;F)$ is a skew field.
  Set $\Sigma = \Sigma(FG \subseteq \calu(G))$, 
  $\Phi = \Sigma(FH  \subseteq \calu(H))$, and let $\mu \colon FH_{\Phi} \to\cald(H;F)$, 
  $\lambda \colon   FG_{\Sigma} \to \cald(G;F)$ denote the corresponding localization maps.  

  Then  $\cald(G;F)$ is a semisimple artinian ring and agrees with $\calr(G;F)$.
  Furthermore if $\mu$ is an isomorphism,  then so is $\lambda$.
\end{lemma}
\begin{proof}
  Let $\cald(H;F)G$ denote the subring of $\cald(G;F)$ generated by $\cald(H;F)$ and $G$.
  Then Lemma~\ref{lem:Lcrossedproduct} shows that for a suitable crossed product, there is
  an isomorphism $\theta \colon \cald(H;F) * G/H \to \cald(H;F)G$ which extends the
  identity map on $\cald(H;F)$.  This ring has dimension $|G/H|$ over the skew field
  $\cald(H;F)$ and is therefore artinian.  Since every matrix over an
artinian ring is either a zero-divisor or invertible (in particular
ever element is either a zero-divisor or invertible), we see that
$\calr(G;F) = \cald(G;F) = \cald(H;F)G$.
Furthermore by Maschke's Theorem, $\cald(H;F)G$
  semisimple artinian.  Now assume that
  $\mu$ is an isomorphism.  We may identify $FG$ with the subring
$FH*G/H$ and then by~\cite[Lemma 4.5]{Linnell(1993)},
there is an isomorphism
  $\psi\colon \cald(H;F) * G/H \to FG_{\Phi}$ which extends the identity map on $FG$.
  Also $\Phi \subseteq \Sigma$, so the identity map on $FG$ extends to a map 
  $\rho \colon   FG_{\Phi} \to FG_{\Sigma}$.  Then 
  $\rho \circ  \psi \circ \theta^{-1} \circ \lambda \colon FG_{\Sigma} \to   FG_{\Sigma}$ is a map extending the 
  identity on $FG$, hence is the identity and the  result follows.
\end{proof}

Recall that the group $G$ is locally indicable if for every a non-trivial finitely generated subgroup
$H$ there exists $N\lhd H$ such that $N/H$ is infinite cyclic.
Also if $R$ is a subring of the skew field $D$ such that
$\mathcal{D}(R \subseteq D) = D$, then we say that $D$ is a field of
fractions for $R$ ($D$ will be noncommutative, i.e.~a skew field in
general).

\begin{definition}
Let $K$ be a skew field, let $G$ be a locally indicable group,
let $K*G$ be a crossed product, and let $D$ be a field of fractions for
$K*G$.  Then we say that
$D$ is a \emph{Hughes-free} \cite[\S 2]{Hughes(1970)},
\cite[pp.~340, 342]{Lewin(1974)},
\cite[Lemma~10.81]{Lueck(2002)},
\cite[p.~1128]{Dicks-Herbera-Sanchez(2004)} field of fractions
for $K*G$ if whenever $N\lhd H \le G$, $H/N$ is infinite cyclic
and $t \in H$ such that $\langle Nt\rangle = H/N$ (i.e.~$t$ generates
$H$ mod $N$), then $\{t^i \mid i \in \mathbb{Z}\}$ is linearly
independent over $\mathcal{D}(K*N \subseteq D)$.
\end{definition}

A key result here is that of Ian Hughes \cite[Theorem]{Hughes(1970)},
\cite[Theorem~7.1]{Dicks-Herbera-Sanchez(2004)}, which states
\begin{theorem}[Hughes's theorem] \label{thm:Hughes}
Let $K$ be a skew field, let $G$ be a locally
indicable group, let $K*G$ be a crossed product, and let $D_1$ and $D_2$
be Hughes-free field of fractions for $K*G$. Then there is an
isomorphism $D_1 \to D_2$ which is the identity on $K*G$.
\end{theorem}

Recall that a ring $R$ is called a \emph{fir} (free ideal ring,
\cite[\S 1.6]{Cohn(1995)}) if every left ideal is a free left
$R$-module of unique rank, and every right ideal is a free right
$R$-module of unique rank.  Also, $R$ is call a semifir if the above
condition is only satisfied for all finitely generated left and right
ideals.  It is easy to see that if $K$ is a skew field,
$G$ is the infinite cyclic group and $K*G$ is a crossed product, then
every nonzero left or right ideal is free of rank one and hence $K*G$
is a fir.  We can now apply \cite[Theorem 5.3.9]{Cohn(1995)} (a
result essentially due to Bergman \cite{Bergman(1974)}) to deduce that
if $G$ is a free group and $K*G$ is a crossed product, then $K*G$ is
a fir.

We also need the concept of a \emph{universal} field of fractions;
this is described in \cite[\S 7.2]{Cohn(1985)} and
\cite[\S 4.5]{Cohn(1995)}.  It is proven in \cite[Corollary
7.5.11]{Cohn(1985)} and \cite[Corollary 4.5.9]{Cohn(1995)} that if
$R$ is a semifir, then is has a universal field of fractions $D$.
Furthermore the inclusion $R \subseteq D$ is an honest map
(\cite[p.~250]{Cohn(1985)}, \cite[p.~177]{Cohn(1995)}), fully
inverting (\cite[p.~415]{Cohn(1985)}, \cite[p.~177]{Cohn(1995)}), and
the localization map $R_{\mathcal{D} (R\subseteq D)} \to
D$ is an isomorphism.  We can now state a crucial result of Jacques
Lewin \cite[Proposition 6]{Lewin(1974)}.
\begin{theorem}[Lewin's theorem] \label{thm:Lewin}
Let $K$ be a skew field, let $G$ be a free group, let $K*G$ be a
crossed product, and let $D$ be the universal field of fractions for
$K*G$.  Then $D$ is Hughes-free.
\end{theorem}
Actually Lewin only proves the result for $K$ a field and $K*G$ the
group algebra $KG$ over $K$.  However with the remarks above, in
particular that $K*G$ is a fir, we can follow Lemmas 1--6 and Theorem
1 of Lewin's paper \cite{Lewin(1974)} verbatim to deduce
Theorem~\ref{thm:Lewin}.

\begin{lemma} \label{lem:LA} Let $H\lhd G$ be groups and let $G/H
 \in \cala$.  Assume that   $\cald(G;F)$ is a skew field. Write
  $\Sigma = \Sigma(FG \subseteq \calu(G))$ and $\Phi = \Sigma(FH \subseteq  \calu(H))$.  
   Let $\mu \colon FH_{\Phi} \to \calr(H;F)$ and $\lambda \colon FG_{\Sigma}
  \to \calr(G;F)$ be the localization maps which extend the identity on $FH$ and $FG$
  respectively.   Suppose that $\mu$ is an isomorphism.

  Then $\cald(G;F) = \calr(G;F)$, and $\lambda$ is an isomorphism.
\end{lemma}
\begin{proof}
  We already know that $\cald(G;F) = \calr(G;F)$ because we are
assuming that $\cald(G;F)$ is a skew field, and clearly
$\lambda$ is an epimorphism.  We need to show that $\lambda$ is
injective.
Lemma~\ref{lem:Lcrossedproduct}
shows that $\cald(H;F) G \cong \cald(H;F)*G/H$ and we will use the
corresponding isomorphism to identify these two rings without further
comment.  Since we are assuming that $\cald(G;F)$ is a skew field,
$\cald(H;F) *G/H$ is a domain.
Furthermore $FG_{\Phi} \cong (FH*G/H)_{\Phi} \cong FH_{\Phi} *G/H$
by Lemma~\ref{lem:Lauto} and \cite[Lemma~4.5]{Linnell(1993)}, 
and we deduce that the localization map $FG_{\Phi} \to \cald(H;F)
*G/H$ is an isomorphism, because we are assuming that $\mu$ is an
isomorphism.
Let $\Psi = \Sigma(\cald(H;F)G \subseteq \cald(G;F))$.
The proof of~\cite[Theorem 4.6]{Schofield(1985)} shows that
  $(FG_{\Phi})_{\Psi} \cong FG_{\Sigma'}$ for a suitable set of
matrices $\Sigma'$ over $FG$ (where we have identified $FG_{\Phi}$
with $\cald(H;F) G$ by the above isomorphisms).
All the matrices in $\Sigma'$ become invertible over
$\calr(G;F)$, so
  by~\cite[Exercise~7.2.8]{Cohn(1985)} we may replace $\Sigma'$ by
its saturation.
It remains to prove that the localization map
$\cald(H;F)G_{\Psi} \to \calr(G;F)$ is injective.

We have two cases to consider, namely $G/H$
  amenable and $G/H$ finitely generated free.  
For the former we 
   apply~\cite[Theorem~6.3]{Dodziuk-Linnell-Mathai-Schick_Yates(2003)}
(essentially a result of Tamari~\cite{Tamari(1957)}).  We deduce that
$\cald(H;F) *G/H$ satisfies the Ore condition for the
multiplicatively closed subset of nonzero elements of
$\cald(H;F)*G/H$ and it follows that the localization map
$\cald(H;F)G_{\Psi} \to \calr(G;F)$ is an isomorphism.

For the latter case, let $L \lhd M$ be subgroups of $G$ containing
$H$ such that $M/L$ is infinite cyclic and let $t \in M$ be a
generator for $M$ mod $L$.  Since the sum $\sum_{i \in \mathbb{Z}}
\mathcal{U}(L)t^i$ is direct, we see that the sum $\sum_{i\in
\mathbb{Z}} \cald(L;F)t^i$ is also direct and we deduce that
$\cald(G;F)$ is a Hughes-free field of fractions for $\cald(H;F)*G/H$.  It now follows from 
Theorems~\ref{thm:Hughes} and~\ref{thm:Lewin} that $\mathcal{D}(G;F)$ is a universal field of
fractions for $\mathcal{D}(H;L)G$ and in particular the localization
map $\cald(H;F)G_{\Psi} \to \calr(G;F)$ is injective.  This finishes
the proof.  \end{proof}

\begin{theorem} \label{the:Linnells_theorem} Let $H\lhd G$ be groups with 
  $H \in  \calx$, $H$ torsionfree and $G/H$ finite.  Let 
  $\Sigma = \Sigma(FG \subseteq   \calu(G))$.  Assume that 
  $\cald(H;F)$ is a skew field.
  
  Then $\cald(G;F) = \calr(G;F)$, and $H$ has the property (UL) with respect to $F$, i.e., 
  the localization map   $FG_{\Sigma} \to \calr(G;F)$ is an isomorphism.
\end{theorem}
\begin{proof}
  We first consider the special case $G=H$ (so $G$ is torsionfree).  We use the
  description of the class of groups $\calx$ given in
  Lemma~\ref{lem:Ldescription}~\eqref{lem:Ldescription:b} and prove the result by
  transfinite induction.  The result is obvious if $G \in \calx_0$, because then $G = 1$. The induction
  step is done as follows.  Consider an ordinal $b$ with $b \not=0$ and a group 
  $G \in  \calx_b$ such that the claim already known for all groups $H \in \calx_a$
  for all ordinals $a < b$.  We have to show the claim for $G$.  If $b$ is a limit ordinal,
  this is obvious since $G$ belongs to $\calx_a$ for every ordinal $a < b$.  It
  remains to treat the case where $b$ is not a limit ordinal.  Then 
  $G \in   \Loc(\calx_a \cala)$ for some ordinal $a < b$.  By Lemma~\ref{lem:Llocally},
it is sufficient to consider the case $G \in \calx_a\cala$.  Now apply
  Lemma~\ref{lem:LA}.

  The general case when $G$ is not necessarily equal to $H$ now follows from
  Lemma~\ref{lem:Lfiniteindex}.
\end{proof}

There are many groups for which Theorem 2.14 can be applied, some of
which we now describe.  Let $N$ be either an Artin pure braid group,
or a RAAG, or a subgroup of finite index in a right-angled
Coxeter group.  Let $\overline{\mathbb{Q}}$ denote the field of all
algebraic numbers.  We can now state
\begin{theorem}
Let $G$ be a group which contains $N$ as a normal subgroup such that
$G/N$ is elementary amenable, and let $\Sigma = \Sigma(FG \subseteq
\calu(G))$.  Assume that $G$ contains a torsionfree
subgroup of finite index  and that $F$ is a subfield of
$\overline{\mathbb{Q}}$.
Then the localization map $FG_{\Sigma} \to \calr(G;F)$ is an
isomorphism, i.e.~$G$ has property (UL) with respect to $F$.
\end{theorem}
\begin{proof}
First we recall some group theoretic results.
An Artin pure braid group is poly-free, see
e.g.~\cite[\S 2.4]{Rolfsen(2010)}, and RAAG's are poly-free by
\cite[Theorem A]{Hermiller-Sunic(2007)}.  Finally right-angled
Coxeter groups have a characteristic subgroup of index a power of
2 which is isomorphic to a subgroup of a right-angled Artin group
\cite[Proposition 5~(2)]{Linnell-Okun-Schick(2012)} and therefore
this subgroup is poly-free.  This shows that
in all cases $G \in \mathcal{X}$
and hence any subgroup of $G$ is in $\mathcal{X}$, because
$\mathcal{X}$ is subgroup closed by
Lemma~\ref{lem:Ldescription}~\eqref{lem:Ldescription:c}.

Now let $H$ be a torsionfree normal subgroup of finite index in $G$.
We need to show that $H$ satisfies the Atiyah conjecture with
coefficients in $F$.  We may assume that $F = \overline{\mathbb{Q}}$.
For the case $N$ is an Artin pure braid group, this
follows from \cite[Corollary 5.41]{Linnell-Schick(2007)}.  For the
case $N$ is a RAAG, this follows from \cite[Theorem
2]{Linnell-Okun-Schick(2012)}.  Finally if $N$ is a subgroup of finite
index in a right-angled Coxeter group, this follows from
\cite[Theorem 2 and Proposition 5 (2)]{Linnell-Okun-Schick(2012)} and
\cite[Theorem 1.1]{Schreve(2014)}.
\end{proof}


\subsection{The property (ULA)}
\label{subsec:The_property_(ULA)}

\begin{definition}[Property (ULA)]
We say that a torsionfree  group $G$ has the property (ULA) with respect to the subfield $F \subseteq \IC$,
if the canonical epimorphism
\[
\lambda \colon R_{\Sigma(FG \subseteq \calr(G;F))} \to \calr(G;F)
\]
is bijective, and $\cald(G;F)$ is a skew field.
\end{definition}

Given a torsionfree group $G$, recall from
Theorem~\ref{the:Status_of_the_Atiyah_Conjecture}~%
\eqref{the:Status_of_the_Atiyah_Conjecture:skew_field} that $\cald(G;F)$ is a skew field
if and only $G$ satisfies the Atiyah Conjecture with coefficients in $F$ and that we have
$\cald(G;F) = \calr(G;F)$ provided that $\cald(G;F)$ is a skew field.  So $G$ satisfies
condition  (ULA) with respect to $F$ if and only if $G$ satisfies both condition (UL) with 
respect to $F$ and the Atiyah Conjecture with coefficients in $F$.

\begin{theorem}[Groups in $\calc$ have property (ULA)]
\label{the:groups_calc_have_property_(ULA)}
  Let $\calc$ be the smallest
  class of groups which contains all free groups and is closed under directed unions and
  extensions with elementary amenable quotients.  Suppose that $G$ is a torsionfree group
  which belongs to $\calc$. 

  Then $G$ has property (ULA).
\end{theorem}
\begin{proof}
This follows from 
Theorem~\ref{the:Status_of_the_Atiyah_Conjecture}~\eqref{the:Status_of_the_Atiyah_Conjecture:skew_field}
and~\eqref{the:Status_of_the_Atiyah_Conjecture:Linnell} 
and Theorem~\ref{the:Linnells_theorem} since obviously $\calc \subseteq \calx$.
\end{proof}


\typeout{------------------------   Section 3: Proof of the main Theorem Groups  ----------------------}

\section{Proof of the main Theorem~\ref{the:Whw(G)_and_units_in_cald(G;IQ)}}
\label{sec:Proof_of_the_main_Theorem}

Next we explain why we are interested in group with properties (ULA) by proving our main 
Theorem~\ref{the:Whw(G)_and_units_in_cald(G;IQ)} which will be a direct consequence
of Theorems~\ref{the:groups_calc_have_property_(ULA)} 
and~\ref{the:K_1_w(ZG)_and_K_1(cald(G))_for_groups_with_property_(ULA)}.

\begin{definition}[$K^w_1(RG)$]
\label{def:K_1_prime(R,Sigma)} 
Let $G$ be a group, let $R$ be a ring with $\IZ \subseteq R \subseteq \IC$,
and denote by $F \subseteq \IC$ its quotient field.  Let
\[
K_1^w(RG)
\]
 be the abelian group defined in terms of generators and
relations as follows.  Generators $[f]$ are given by (conjugacy classes) of $RG$-endomorphisms $f \colon P \to P$ 
of finitely generated projective $RG$-modules $P$ such that 
$\omega_* f  \colon \omega_* P \to \omega_* P$ is a $\cald(G;F)$-isomorphism
for the inclusion $\omega \colon RG \to \cald(G;F)$.  
If $f,g \colon P \to P$ are $RG$-endomorphisms
of the same finitely generated projective $RG$-module $P$ such that $\omega_* f$ and 
$\omega_* g$ are bijective, then we require the relation
\[
[g \circ f] = [g] + [f].
\]
If we have a commutative diagram of finitely generated projective $RG$-modules
with exact rows
\[
\xymatrix{0 \ar[r] 
&
P_0 \ar[r]^{i} \ar[d]_{f_0}
& 
P_1 \ar[r]^{p} \ar[d]_{f_1}
& 
P_2 \ar[r] \ar[d]_{f_2}
&
0
\\
0 \ar[r] 
&
P_0 \ar[r]^{i} 
& 
P_1 \ar[r]^{p}
& 
P_2 \ar[r]
&
0
}
\]
such that $\omega_*  f_0$, $\omega_* f_1$, and $\omega_* f_2$  are bijective, then we require
the relation
\[
[f_1] = [f_0] + [f_2].
\]
Furthermore, define
\begin{eqnarray*}
\widetilde{K}_1^w(RG) 
& := & 
\coker\bigl( \{\pm 1 \}  \xrightarrow{\cong} K_1(\IZ) \to K_1(\IZ G) \to K_1^w(RG)\bigr);
\\
\Wh^w(G;R) 
& = & 
\coker\bigl( \{\pm g \mid g \in G\} \to K_1(\IZ G) \to K_1^w(RG)\bigr);
\\
\Wh^w(G) 
& = & \Wh^w(G;\IZ); 
\\
\widetilde{K}_1(\calr(G;F)) 
& := & 
\coker\bigl( \{\pm 1 \} \xrightarrow{\cong} K_1(\IZ) \to K_1(\IZ G) \to K_1(\calr(G;F))\bigr);
\\
\Wh(\calr(G;F)) 
& = & 
\coker\bigl( \{\pm g \mid g \in G\} \to K_1(\IZ G) \to K_1(\calr(G;F))\bigr).
\end{eqnarray*}
\end{definition}

\begin{remark} Let $A$ be a square matrix over $RG$. Then the square matrix $\omega(A)$
  over $\cald(G;F)$ is invertible if and only if the operator 
  $r_A^{(2)} \colon  L^2(G)^n \to L^2(G)^n$ given by right multiplication with $A$ is a weak isomorphism,
  i.e., it is injective and has dense image.  This follows from the conclusion
  of~\cite[Theorem~6.24 on page~249 and Theorem~8.22~(5) on page~327]{Lueck(2002)} that
  $r_A^{(2)}$ is a weak isomorphisms if and only if it becomes invertible in $\calu(G)$.
\end{remark}

There is a Dieudonn\'e determinant for invertible matrices over a skew field $D$ which takes
values in the abelianization of the group of units $D^{\times}/[D^{\times},D^{\times}]$ 
and induces an isomorphism, see~\cite[Corollary~43 on page~133]{Silvester(1981)} 
\begin{eqnarray}
{\det}_D \colon K_1(D) 
& \xrightarrow{\cong}  &
D^{\times}/[D^{\times},D^{\times}]. 
\label{K_1(K)_Dieudonne}
\end{eqnarray}
The inverse 
\begin{eqnarray}
J_D \colon  D^{\times}/[D^{\times},D^{\times}] & \xrightarrow{\cong} & K_1(D) 
\label{K_1(K)_Dieudonne_inverse}
\end{eqnarray}
sends the class of a unit in $D$ to the class of the corresponding $(1,1)$-matrix.

\begin{theorem}[$K_1^w(FG)$ for groups with property {(ULA) with respect to $F$}]
\label{the:K_1_w(ZG)_and_K_1(cald(G))_for_groups_with_property_(ULA)}
Let $R$ be a ring with $\IZ \subseteq R\subseteq \IC$. Denote by $F \subseteq \IC$ 
the quotient field of $R$.  Let $G$ be a torsionfree group with the property (ULA) with
respect to $F$.

Then the canonical maps
sending $[f]$ to $[\omega_*f]$
\begin{eqnarray*}
\omega_*  \colon K_1^w(RG) & \xrightarrow{\cong} & K_1(\cald(G;F));
\\
\omega_*  \colon \widetilde{K}_1^w(RG) & \xrightarrow{\cong} & \widetilde{K}_1(\cald(G;F));
\\
\omega_*  \colon \Wh^w(G;R) & \xrightarrow{\cong} & \Wh(\cald(G;F)).
\end{eqnarray*}
are bijective. Moreover, $\cald(G;F)$ is a skew field and the Dieudonne determinant
induces an isomorphism
\[
{\det}_D \colon K_1(\cald(G;F)) \xrightarrow{\cong} \cald(G;F)^{\times}/[\cald(G;F)^{\times},\cald(G;F)^{\times}].
\]
\end{theorem}
\begin{proof}
This follows directly from Theorem~\ref{the:K_1(R,Sigma)_and_K_1(R_Sigma)}.
\end{proof}

Finally we can give the proof of Theorem~\ref{the:Whw(G)_and_units_in_cald(G;IQ)}.

\begin{proof}[Proof of Theorem~\ref{the:Whw(G)_and_units_in_cald(G;IQ)}]
  Because of Theorem~\ref{the:groups_calc_have_property_(ULA)} the group $G$ has property
  (ULA) and we can apply Theorem~\ref{the:K_1(R,Sigma)_and_K_1(R_Sigma)}. It remains to
  explain why in the special case $R = \IZ$ the group $K_1^w(\IZ G)$ as appearing in
  Theorem~\ref{the:K_1(R,Sigma)_and_K_1(R_Sigma)}, namely, as introduced in
  Definition~\ref{def:K_1_prime(R,Sigma)} agrees with the group $K_1^w(\IZ G)$ appearing
  in the introduction. This boils down to explain why for a $(n,n)$-matrix $A$ over $\IZ   G$ 
  the operator $r_A^{(2)} \colon L^2(G)^n \to L^2(G)^n$ is a weak isomorphism if and
  only if $A$ becomes invertible in $\cald(G;\IQ)$. By definition $A$ is invertible in $\cald(G;\IQ)$
  if and only if it is invertible in $\calu(G)$. Now 
  apply~\cite[Theorem~6.24 on page 249 and Theorem~8.22~(5) on page~327]{Lueck(2002)}.
  \end{proof}





\typeout{-------------------------------------- References  --------------------------------------}



\end{document}